\documentclass{article}

\usepackage{url}

\usepackage{amsmath,amsfonts,amssymb}
\usepackage{amsthm}
\usepackage{enumerate}

\usepackage[dvips]{color}

\definecolor{refkey}{gray}{0.75}
\definecolor{labelkey}{gray}{0.75}

\newtheorem{theorem}{Theorem}[section]
\newtheorem{proposition}[theorem]{Proposition}

\theoremstyle{definition}

\newtheorem{example}[theorem]{Example}

\newtheorem{remark}[theorem]{Remark}

\numberwithin{equation}{section}

\newcommand{\nC}{\mathbb C}

\newcommand{\nN}{\mathbb N}
\newcommand{\nR}{\mathbb R}
\newcommand{\nZ}{\mathbb Z}

\newcommand{\nJ}{\mathbb J}

\newcommand{\mc}[1]{{\mathcal #1}}
\newcommand{\co}[1]{\overline{#1}}

\DeclareMathOperator{\vsum}{sum} %
\DeclareMathOperator{\var}{var} %

\begin{document}

\title%
{Somewhat Stochastic Matrices}

\author{Branko \'{C}urgus and Robert I. Jewett}
\date{}


\maketitle


\section{Introduction.}

The notion of a Markov chain is ubiquitous in linear algebra and
probability books. For example see \cite[Theorem~5.25]{FIS} and
\cite[p.~173]{Do}. Also, see \cite[p.~131]{S} for the history of the
subject. A {\em Markov} (or {\em stochastic}) {\em matrix} is a
square matrix whose entries are non-negative and whose column sums
are equal to $1$. The term stochastic matrix seems to prevail in
current literature and therefore we use it in the title. But, since
a Markov matrix is a transition matrix of a Markov chain, we prefer
the term Markov matrix and we use it from now on. The theorem below
gives some of the standard results for such matrices.

\begin{theorem} \label{tmti}
Let $M$ be an $n\!\times\!n$ Markov matrix. Suppose that there exists
$p \in \nN$ such that all the entries of $M^p$ are positive. Then the
following statements are true.
\begin{enumerate}[{\rm (a)}]
\item \label{itmia}
There exists a unique $E \in \nR^n$ such that
\begin{equation*}
ME = E \ \ \ \ \text{and}  \ \ \  \ \vsum E = 1.
\end{equation*}
\item \label{itmib}
Let $P$ be the square matrix each of whose columns is equal to $E$.
Then $P$ is a projection and $PX = (\vsum X) E$ for each $X \in
\nR^n$.
\item \label{itmic}
The powers $M^k$ tend to $P$ as $k$ tends to $+\infty$.
\end{enumerate}
\end{theorem}

The statement that all the entries of some power of $M$ are positive
is usually abbreviated by saying that $M$ is {\em regular}. The fact
that all the entries of $E$ are positive is easily shown, since $M^p
E = E, \vsum E = 1$, and $M^p$ is positive.

Theorem~\ref{tmti} follows readily from Theorem~\ref{tmt1}, the main
result in this article. In Theorem~\ref{tmt1} the requirement that
the entries of $M$ be non-negative is dropped, the requirement that
the column sums be equal to $1$ is retained, and the condition on
$M^p$ is replaced by something completely different. However, the
conclusions (\ref{itmia}), (\ref{itmib}), (\ref{itmic}) hold true.
Our proof is significantly different from all proofs of
Theorem~\ref{tmti} that we are aware of.

Here is an example:
\begin{equation*}
M = \frac{1}{5} \begin{bmatrix} \phantom{-}0 & \phantom{-}2 & -4 \\
               -1 & -1 & \phantom{-}0 \\
                \phantom{-}6 & \phantom{-}4 & \phantom{-}9
                \end{bmatrix}.
\end{equation*}
Examining the first ten powers of $M$ with a computer  strongly
suggests that the powers of $M$ converge. Indeed, Theorem~\ref{tmt1}
applies here. For this, one must examine $M^2$; see
Example~\ref{exeI}.  The limit is found, as in the case of a Markov
matrix, by determining an eigenvector of $M$. It turns out that $ E =
\frac{1}{3} \begin{bmatrix} -6 & 1 & 8
                \end{bmatrix}^{{\text{\tiny ${\mathsf T}$}}},
$
and
\begin{equation*}
\lim_{k\to +\infty}  M^k  = \frac{1}{3} \begin{bmatrix} -6 & -6 & -6 \\
               \phantom{-}1 & \phantom{-}1 & \phantom{-}1 \\
                \phantom{-}8 & \phantom{-}8 & \phantom{-}8
                \end{bmatrix}.
\end{equation*}
Since $M$ is only a $3\!\times\!3$ matrix, we could show convergence
by looking at the eigenvalues, which are $1, 2/5, 1/5$.

Theorem~\ref{tmti} is often presented as an application of the
Perron-Frobenius Theorem. In \cite{TRH}, the authors give a version
of the Perron-Frobenius Theorem for matrices with some negative
entries, but their results do not seem to be related to ours.

\section{Definitions.}

All numbers in this article are real, except in Example~\ref{execm}.
We study $m\!\times\!n$ matrices with real entries.  The elements of
$\nR^n$ will be identified with column matrices, that is, with
$n\!\times\!1$ matrices. By $J$ we denote any row matrix with all
entries equal to $1$.

Let $X \in \nR^n$ with entries $x_1, \ldots, x_n$.  Set
\begin{equation*}
\vsum X  := \sum_{j=1}^n x_j,  \ \ \ \ \ \|X\|  := \sum_{j=1}^n
|x_j|.
\end{equation*}
Notice that $\vsum X = J X$ and that $\|\cdot\|$ is the
$\ell_1$-norm.

For an $m\!\times\!n$ matrix $A$ with columns $A_1, \ldots, A_n$ the
{\em variation} (or {\em column variation}) of $A$ is defined by:
\begin{equation*}
\var A  := \frac{1}{2} \ \max_{\substack{1 \leq j, \, k \leq n}}
\bigl\| A_j - A_k\bigr\|.
\end{equation*}

If the column sums of a matrix $A$ are all equal to $a$, that is if
$J A = a J$, we say that $A$ is of {\em type}\, $a$.

\section{Column variation and matrix type.}

In this section we establish the properties of the column variation
and the matrix type that are needed for the proof of our main
result. We point our that the restriction to real numbers in the
theorem below is essential, as Example~\ref{execm} shows.

\begin{theorem} \label{ileq}
Let $A$ be an $m\!\times\!n$ matrix and $X \in \nR^n$. If $\vsum X =
0$, then
\[
\|AX\| \leq (\var A) \|X\|.
\]
\end{theorem}
\begin{proof}
In this proof, for any real number $t$ we put $t^+ := \max\{t,0\}$
and $t^- := \max\{-t,0\}$. Clearly $t^+, t^- \geq 0$ and $t = t^+ -
t^-$.

Let $A_1, \ldots, A_n$ be the columns of $A$. Assume $\vsum X =0$.
The conclusion is obvious if $X = 0$.

Assume that $X \neq 0$. Then, by scaling, we can also assume that
$\|X\| = 2$. Let $x_1, \ldots, x_n$ be the entries of $X$. Then
\begin{equation*}
\sum_{k=1}^n x_k^+ = \sum_{k=1}^n x_k^- = 1.
\end{equation*}
Consequently
\begin{equation*}
 AX   =  \sum_{k=1}^n x_k A_k
  =  \ \sum_{k=1}^n x_k^+ A_k  \ - \ \sum_{j=1}^n x_j^- A_j.
\end{equation*}
Now we notice that $AX$ is the difference of two convex combinations
of the columns of $A$. From this, the inequality in the theorem seems
geometrically obvious. However, we continue with an algebraic
argument:
\begin{align*}
  AX  & =
  \sum_{k=1}^n \sum_{j=1}^n x_j^- x_k^+ A_k
  - \sum_{j=1}^n \sum_{k=1}^n x_k^+ x_j^- A_j   \\
 & =
  \sum_{k=1}^n \sum_{j=1}^n  x_j^- x_k^+ \bigl(A_k -A_j \bigr).
 \end{align*}
Consequently,
\begin{align*}
\| AX \|
 & \leq
  \sum_{k=1}^n \sum_{j=1}^n x_j^- x_k^+ \bigl\| A_k -A_j \bigr\|
    \\
 & \leq 2 (\var A)
  \sum_{k=1}^n x_k^+ \sum_{j=1}^n x_j^-
    \\
  & = (\var A) \,  \| X \|. \qedhere
\end{align*}
\end{proof}

\begin{proposition} \label{ivarAB}
Let $A$ and $B$ be matrices such that $AB$ is defined.  If $B$ is of
type $b$, then
\begin{equation*}
\var(AB) \leq (\var A) (\var B).
\end{equation*}
\end{proposition}
\begin{proof}
Assume that $B$ is of type $b$ and let $B_1, \ldots, B_l$ be the
columns of $B$.  Then $AB_1, \ldots, AB_l$ are the columns of $AB$.
Since $B$ is of type $b$, for all $j, k \in \{1,\ldots,l\}$ we have
$\vsum(B_j-B_k) = 0$. Therefore, by Theorem~\ref{ileq},
\begin{equation*}
\| AB_j - AB_k \| \leq (\var A) \, \|B_j - B_k \|
\end{equation*}
 for all $j, k \in \{1,\ldots,l\}$.  Hence,
\begin{align*}
\var(AB) & = \frac{1}{2} \, \max_{\substack{1 \leq j, \, k \leq l}}
\| AB_j - AB_k
\| \\
 & \leq (\var A) \,\frac{1}{2} \, \max_{\substack{1 \leq j, \, k \leq l}}   \|B_j -
B_k \| \\
& = (\var A) (\var B). \qedhere
\end{align*}
\end{proof}

\begin{proposition}  \label{ityAB}
Let $A$ and $B$ be matrices such that $AB$ is defined.  If $A$ is of
type $a$ and $B$ is of type $b$, then $AB$ is of type $ab$.
\end{proposition}
\begin{proof}
If $J A = a J$ and $J B = b J$, then $J (AB) = (J A)B = a J B = ab
J$.
\end{proof}

\section{Square matrices.}

In the previous section we considered rectangular matrices. Next we
study square matrices. With one more property of matrix type, we
shall be ready to prove our main result, Theorem~\ref{tmt1}.

\begin{proposition} \label{ievc}
If $M$ is a square matrix of type $c$, then $c$ is an eigenvalue of
$M$.
\end{proposition}
\begin{proof}
Assume that $J M = c J$. Then $J (M-cI) = 0$. That is, the sum of the
rows of $M-cI$ is $0$ and so the rows of $M-cI$ are linearly
dependent. Hence, $M-cI$ is a singular matrix.
\end{proof}

\begin{theorem} \label{tmt1}
Let $M$ be an $n\!\times\!n$ matrix. Suppose that $M$ is of type $1$
and that there exists $p \in \nN$ such that $\var (M^p) < 1$. Then
the following statements are true.
 \begin{enumerate}[{\rm (a)}]
 \item \label{itma}
 There exists a unique $E \in \nR^n$ such that
 \begin{equation*}
 ME = E \ \ \ \ \text{and}  \ \ \  \ \vsum E = 1.
 \end{equation*}
 \item \label{itmb}
 Let $P$ be the square matrix each of whose columns is equal to $E$.
 Then $P$ is a projection and $PX = (\vsum X) E$ for each $X \in
 \nR^n$.
 \item \label{itmc}
 The powers $M^k$ tend to $P$ as $k$ tends to $+\infty$.
 \end{enumerate}
\end{theorem}

\begin{proof}
Assume that $M$ is of type $1$ and that there exists $p \in \nN$ such
that $\var (M^p) < 1$. By Proposition~\ref{ievc}, there exists a
nonzero $Y \in \nR^n$ such that $MY =Y$.

Clearly $M^p Y=Y$. If $\vsum Y =0$, then, since $Y \neq 0$,
Theorem~\ref{ileq} yields
\begin{equation*} 
\|Y\| = \| M^p Y \| \leq \var(M^p) \|Y\| < \|Y\|,
\end{equation*}
a contradiction. Setting $E = (1/\vsum Y)Y$ provides a vector whose
existence is claimed in (\ref{itmia}). To verify uniqueness, let $F$
be another such vector.  Then $\vsum(E-F) = 0,$ $M^p(E-F)=E-F$, and
\begin{equation*} 
\|E-F\| = \| M^p (E-F) \| \leq \var(M^p) \|E-F\|.
\end{equation*}
Consequently,  $E-F =0$, since $\var (M^p) < 1$.

By the definition of $P$ in (\ref{itmib}), $P = EJ$. Therefore, $P^2
= ( EJ)( EJ) =  E (J E) J = E [1] J = E J = P$.  To complete the
proof of (\ref{itmib}), we calculate: $PX = E (J X) = (\vsum X) E$.

Let $k \in \nN$. Proposition~\ref{ityAB} implies that $M^k$ is of
type $1$. By the division algorithm there exist unique $q, r \in \nZ$
such that $k = p q + r$ and $0 \leq r < p$. Here $q = \lfloor k/p
\rfloor$ is the floor of $k/p$. By Proposition~\ref{ivarAB},
\begin{equation} \label{eqtml0}
\var( M^k) \leq (\var M)^r \, \bigl( \var(M^p)\bigr)^q \leq \Bigl(
\max_{\substack{0 \leq r < p}} \ (\var M)^r \!\Bigr) \bigl(
\var(M^p)\bigr)^{\lfloor k/p \rfloor}.
\end{equation}
Let $X \in \nR^n$ be such that $\vsum X = 1$. Then $\vsum(X-E) = 0$
and Theorem~\ref{ileq} implies that
\begin{equation} \label{eqtmni}
\bigl\| M^k X - E \bigr\| = \bigl\| M^k (X - E) \bigr\| \leq
\var(M^k) \, \bigl\| X - E  \bigr\|.
\end{equation}

Now, since $\var (M^p) < 1$, \eqref{eqtml0} implies that
\begin{equation*} 
\displaystyle \lim_{k\to+\infty} \var\bigl(M^k\bigr) = 0,
\end{equation*}
and letting $X$ in \eqref{eqtmni} run through the vectors in the
standard basis of $\nR^n$ proves (\ref{itmic}).
\end{proof}

\begin{remark}
If $M$ is an $n\!\times\!n$ matrix of type $1$, and the statements
(\ref{itmia}), (\ref{itmib}), and (\ref{itmic}) are true, then there
exists $p\in \nN$ such that $\var(M^p) < 1$. This follows from the
fact that the variation function is continuous on the space of all
$n\!\times\!n$ matrices.
\end{remark}

\section{Non-negative matrices.}

The propositions in this section are useful for showing that the
variation of a Markov matrix is less than $1$. They are used to
deduce Theorem~\ref{tmti} from Theorem~\ref{tmt1} and also,
repeatedly, in Example~\ref{exmm}.

\begin{proposition} \label{pmm}
Let $A$ be an $m\!\times\!n$  matrix of type $a$ with non-negative
entries. Then $\var A \leq a$.
\end{proposition}
\begin{proof}
Let  $A_1, \ldots, A_n$ be the columns of $A$.  Since the entries of
$A$ are non-negative, $\|A_j\| = a$ for all $j \in \{1,\ldots,n\}$.
Therefore,
\[
\|A_j - A_k \| \leq \|A_j\| + \|A_k\| = 2a
\]
for all $j,k \in \{1,\ldots,n\}$, and the proposition follows.
\end{proof}

\begin{proposition} \label{pmm1}
Let $a > 0$. Let $A$ be an $m\!\times\!n$  matrix of type $a$ with
non-negative entries. Then the following two statements are
equivalent.
\begin{enumerate}[{\rm (i)}]
\item \label{ipmmb}
The strict inequality $\var A < a$ holds.
\item \label{ipmmc}
For each $k, l \in \{1,\ldots,n\}$ there exists $j \in
\{1,\ldots,m\}$ such that the $j$-th entries of the $k$-th and $l$-th
columns of $A$ are both positive.
\end{enumerate}
\end{proposition}
\begin{proof}
Let $A = \bigl[ a_{jk} \bigr]$ and let $A_1, \ldots, A_n$ be the
columns of $A$. To prove that (\ref{ipmmb}) and (\ref{ipmmc}) are
equivalent we consider their negations.  By Proposition~\ref{pmm} and
the definition, $\var A = a$ if and only if there exist $k_0, l_0 \in
\{1,\ldots,n\}$ such that
\[
\bigl\|A_{k_0} - A_{l_0} \bigr\| = 2a.
\]
This is equivalent to
\begin{equation} \label{eqpa1}
\sum_{j=1}^m \bigl| a_{jk_0} - a_{jl_0} \bigr| = \sum_{j=1}^m \bigl(
a_{jk_0} + a_{jl_0} \bigr).
\end{equation}
Since all the terms in the last equality are non-negative,
\eqref{eqpa1} is equivalent to $a_{jk_0}$ or $a_{jl_0}$ being $0$ for
all $j \in \{1,\ldots,m\}.$

 Hence, $\var A = a$ if and only if there
exist $k_0, l_0 \in \{1,\ldots,n\}$ such that for all $j \in
\{1,\ldots,m\}$ we have $a_{jk_0} \, a_{jl_0} = 0$. This proves that
(\ref{ipmmb}) and (\ref{ipmmc}) are equivalent.
\end{proof}

Now we can give a short proof of Theorem~\ref{tmti}.

\begin{proof}
Let $M$ be a regular Markov matrix and assume that $M^p$ is
positive.  By Proposition~\ref{pmm1}, $\var(M^p) < 1$. Therefore
Theorem~\ref{tmt1} applies.
\end{proof}

\section{Examples.} \label{sexe}

\begin{example} \label{exeI}
In the Introduction we used
\begin{equation*}
M = \frac{1}{5} \begin{bmatrix} \phantom{-}0 & \phantom{-}2 & -4 \\
               -1 & -1 & \phantom{-}0 \\
                \phantom{-}6 & \phantom{-}4 & \phantom{-}9
                \end{bmatrix}
\end{equation*}
as an example of a matrix for which the powers converge. The largest
\mbox{$\ell_1$-distance} between two columns is between the second
and the third, and is equal to $12/5$. Therefore, $\var M = 6/5
> 1$. But
\begin{equation*}
M^2 = \frac{1}{25} \begin{bmatrix} -26 & -18 & -36 \\
               \phantom{-2}1 & \,\,\,-1 & \phantom{-3}4 \\
                \phantom{-}50 & \phantom{-}44 & \phantom{-}57
                \end{bmatrix}
\end{equation*}
and $\var(M^2) = 18/25 < 1$. Hence, Theorem~\ref{tmt1} applies.
\end{example}

\begin{example}
For $2\!\times\!2$ matrices of type $1$ it is possible to give a
complete analysis. Let $a,b \in \nR$ and set
\[
M = \left[ \begin{array}{cc} 1-a & b
\\ a & 1-b \end{array}\right] \ \ \ \text{and} \ \ \ c=a+b.
\]
Then $\var M = |1-c|$. We distinguish the following three cases:
\begin{enumerate}[(i)]
\item
$c\neq 0$. The eigenvalues of $M$ are $1$ and $1-c$, and the
corresponding eigenvectors are $\begin{bmatrix} b \\ a
\end{bmatrix}$ and $\begin{bmatrix} 1 \\ -1
\end{bmatrix}$. If $0 < c < 2$, then $\var M < 1$. Consequently, $E =
\dfrac{1}{c} \begin{bmatrix} b \\ a
\end{bmatrix}$ and $M^k$ converges. Otherwise, $M^k$ diverges.

\item
$c=0, \, a \neq 0$. \, In this case, $\var M = 1$ and $1$ is an
eigenvalue of multiplicity $2$. It can be shown by induction that
\[
M^k = \begin{bmatrix} 1 & 0 \\ 0 & 1
\end{bmatrix} + k\, a  \begin{bmatrix} -1 & -1 \\ 1 & 1
\end{bmatrix}.
\]
So $M^k$ diverges.

\item
$c = a = b = 0$. \, So, $M = I$.
\end{enumerate}

Thus, for a $2\!\times\!2$ matrix $M$ of type $1$ which is not the
identity matrix, $M^k$ converges if and only if $\var M < 1$. Regular
$2\!\times\!2$  Markov matrices were studied in \cite{Ro}.

\end{example}

\begin{example} \label{exmm}
Consider the following three kinds of Markov matrices:
\begin{equation*}
K = \begin{bmatrix} 1 & + & + \\ 0 & + & + \\
0 & 0 & + \end{bmatrix}, \ \ \ L = \begin{bmatrix} + & + & 0 \\ + & 0
& + \\ 0 & + & + \end{bmatrix}, \ \ \ M = \begin{bmatrix}
 0 & + & 0 \\
 0 & 0 & 1 \\
 1 & + & 0
  \end{bmatrix}.
\end{equation*}
Here we use $+$ for positive numbers. All claims below about the
variation rely on Propositions~\ref{pmm} and~\ref{pmm1}.

The matrix $K$ is not regular, but $\var K < 1$. Also, $E =
\begin{bmatrix} 1 & 0 & 0 \end{bmatrix}^{{\text{\tiny ${\mathsf T}$}}}$.

The matrix $L$ is not positive, but $\var L < 1$. Also,
Theorem~\ref{tmti} applies since $L^2$ is positive.

The first five powers of $M$ are:
\begin{equation*}
\begin{bmatrix}
 0 & + & 0 \\
 0 & 0 & 1 \\
 1 & + & 0
  \end{bmatrix}, \ \
\begin{bmatrix}
 0 & 0 & + \\
 1 & + & 0 \\
 0 & + & +
   \end{bmatrix},  \ \
\begin{bmatrix}
+ & + & 0 \\
 0 & + & + \\
 + & + & +
   \end{bmatrix}, \ \
\begin{bmatrix}
0 & + & + \\
 + & + & + \\
 + & + & +
\end{bmatrix}, \ \
\begin{bmatrix}
+ &  + & + \\
+ & + & + \\
+ & + & +
\end{bmatrix}.
\end{equation*}
The variation of the first two matrices is $1$, while $\var(M^3) <
1$. The first positive power of $M$ is $M^5$.

In fact, the following general statement holds. For a $3\!\times\!3$
Markov matrix $M$, the sequence $M^k,k\in\nN$, converges to a
projection of rank $1$ if and only if $\var(M^3) < 1$. This was
verified by examining all possible cases; see \cite{CJ}.
\end{example}

\begin{example} \label{execm}
In this example we consider matrices with complex entries. Let
$\alpha = (-1+ i \sqrt{3})/2$. Then $1, \alpha$, and $\co{\alpha}$
are the cube roots of unity. Notice that $1+ \alpha + \co{\alpha} =
0, \ \alpha^2 = \co{\alpha}$, and $\co{\alpha}^2 = \alpha$.

The examples below were suggested by the following orthogonal basis
for the complex inner product space $\nC^3$:
\[
U = \begin{bmatrix} 1 \\
 1 \\ 1 \end{bmatrix}, \ \ \ \ V = \begin{bmatrix} 1 \\
 \alpha \\ \co{\alpha} \end{bmatrix}, \ \ \ \
 W = \begin{bmatrix} 1 \\ \co{\alpha}  \\
 \alpha \end{bmatrix}.
\]

We first give an example which shows that the conclusion of
Theorem~\ref{ileq} does not hold for matrices with complex entries.
Set $ A=
\begin{bmatrix} 1 & \co{\alpha} & \alpha \end{bmatrix}$. Then $AV
=[3], \  \vsum V = 0, \ \var A = {\sqrt{3}}/{2} < 1$, and
\[
 \|AV\| > (\var A) \|V\|.
\]

Next we give an example showing that the restriction to real numbers
cannot be dropped in Theorem~\ref{tmt1}. Consider the matrices
\begin{equation*}
 P = \frac{1}{3} \begin{bmatrix}
     1 & 1 & 1 \\ 1 & 1 & 1 \\1 & 1 & 1
            \end{bmatrix}  \ \ \ \text{and} \ \ \
 Q = \frac{1}{3} \begin{bmatrix}
   1  & \co{\alpha} & \alpha \\
   \alpha & 1 &  \co{\alpha}  \\
    \co{\alpha} & \alpha & 1
   \end{bmatrix}.
\end{equation*}
Notice that $P$ is the orthogonal projection onto the span of $U$ and
$Q$ is the orthogonal projection onto the span of $V$. Let $c \in
\nR$ and set
\[
M = P+c\,Q.
\]
Then $PV = 0$ and $QV  = V$. Therefore $MV = c\, V$, showing that $c$
is an eigenvalue of $M$.

The matrix $P$ is of type $1$ with variation $0$, while $Q$ is of
type $0$ with variation $\sqrt{3}/2$. Hence, $M$ is of type $1$ and
\[
\var M = \var(c\,Q) = |c| \sqrt{3}/2.
\]
Therefore, if $1 < c < 2/\sqrt{3}$, then $\var M < 1$, but $M^k$
diverges.
\end{example}

\section{The variation as a norm.}

The first proposition below shows that the variation function is a
pseudo-norm on the space of all $m\!\times\!n$ matrices. The
remaining propositions identify the variation of a matrix as the
norm of a related linear transformation.

\begin{proposition} \label{ppn}
Let $A$  be $m\!\times\!n$ matrix.
\begin{enumerate}[{\rm (a)}]
\item \label{ippna}
If $c \in \nR$, then $\var(cA) = |c| \var A$.
\item \label{ippnb}
All columns of $A$ are identical if and only if $\var A = 0$.
\item\label{ippnc}
If $B$ is another $m\!\times\!n$ matrix, then $\var(A+B) \leq \var A
+ \var B. $
\end{enumerate}
\end{proposition}
\begin{proof}
The proofs of (\ref{ippna}) and (\ref{ippnb}) are straightforward.
To prove (\ref{ippnc}), let $A_1, \ldots,\!A_n$ be the columns of
$A$ and let $B_1, \ldots, B_n$ be the columns of $B$. Then $A_1+B_1,
\ldots, A_n+B_n$ are the columns of $A+B$, and for all $j,k \in
\{1,\ldots,n\}$,
\[
\| (A_j+B_j) - (A_k+B_k) \| \leq \| A_j - A_k \| + \| B_j - B_k \|.
\qedhere
\]
\end{proof}

\begin{proposition}
Let $A$ be an $m\!\times\!n$ matrix with more then one column.  Then
\begin{equation} \label{eqvn}
\var A = \max\Bigl\{ \| A X \| \ : \ X \in \nR^n, \ \| X \| = 1, \
\vsum X =0 \Bigr\}.
\end{equation}
\end{proposition}
\begin{proof}
It follows from Theorem~\ref{ileq} that the set on the right-hand
side of \eqref{eqvn} is bounded by $\var A$. To prove that $\var A$
is the maximum, let $A_{j_0}$ and $A_{k_0}$ be columns of $A$ such
that $\bigl\|A_{j_0} - A_{k_0}\bigr\| = 2 \var A$.  Choose $X_0 \in
\nR^n$ such that its $j_0$-th entry is $1/2$, its $k_0$-th entry is
$-1/2$ and all other entries are $0$. Then $\| A X_0 \| = \var A$,
$\| X_0 \| = 1$ and $\vsum(X_0) = 0$.
\end{proof}

\begin{remark}
Let $\mc{V}_n$ denote the set of all vectors $X \in \nR^n$ such that
$\vsum X = 0$. This is a subspace of $\nR^n$. An $m\!\times\!n$
matrix $A$ determines a linear transformation from $\mc{V}_n$ to
$\nR^m$. The previous proposition tells us that the norm of this
transformation is $\var A$.
\end{remark}

\begin{proposition} \label{plvn}
Let $B$ be an $m\!\times\!n$ matrix of type $b$ with more then one
row. Then
\begin{equation} \label{eqvln}
\var B = \max\Bigl\{ \var(Z B) \ : \
 Z^{{\text{\tiny ${\mathsf T}$}}}
 \in \nR^m, \ \var Z = 1 \Bigr\}.
\end{equation}
\end{proposition}
\begin{proof}
If $\var B =0$ the statement follows from Proposition~\ref{ivarAB}.
So, assume $\var B > 0$. By Proposition~\ref{ivarAB}, the set on the
right-hand side of \eqref{eqvln} is bounded by $\var B$. Let $B =
\bigl[ b_{jk}\bigr]$ and let $B_{k_0}$ and $B_{l_0}$ be columns of
$B$ such that $\bigl\|B_{k_0} - B_{l_0}\bigr\| = 2 \var B$. Let $Z_0$
be the row with entries defined by:
\begin{equation*}
z_j = \begin{cases} \phantom{-}  1 & \ \ \text{if}
               \ \ \ b_{jk_0} > b_{jl_0}, \\[4pt]
            -1 & \ \ \text{if} \ \ \ b_{jk_0} \leq b_{jl_0}.
\end{cases}
\end{equation*}
Since $\bigl\|B_{k_0} - B_{l_0}\bigr\| > 0$ and $\vsum(B_{k_0}) =
\vsum(B_{l_0}) = 0$, there exist at least one positive and at least
one negative entry in $Z_0$. Therefore, $\var(Z_0) = 1$. By the
definition of $Z_0$ the difference between $k_0$-th and $l_0$-th
entry in $Z_0B$ is $\bigl\|B_{k_0} - B_{l_0}\bigr\| /2 = \var B$.
Notice that if $Z$ is a row matrix, then $\var Z =
\frac{1}{2}\bigl(\max Z - \min Z\bigr)$. Therefore, $\var(Z_0B) \geq
\var B$. This proves \eqref{eqvln}.
\end{proof}

\begin{remark}
Let $\nR_m$ denote the set of all row vectors with $m$ entries.
Denote by $\nJ_m$ the subspace of all scalar multiples of $J$.  An
$m\!\times\!n$ matrix $B$ of type $b$ determines a linear
transformation from the factor space $\nR_m/\nJ_m$ to the factor
space $\nR_n/\nJ_n$ in the following way:
\[
 Z  + \nJ_m \mapsto  ZB  + \nJ_n, \ \ \ Z
\in \nR_m.
\]
It is easy to verify that this is a well-defined linear
transformation.  By Proposition~\ref{plvn}, the norm of this
transformation is exactly $\var B$.
\end{remark}

\noindent\textbf{Branko~\'{C}urgus}

\noindent\textit{Department of Mathematics, Western Washington
University, \\ Bellingham, Washington 98225} \\
{\tt curgus@cc.wwu.edu}

\bigskip

\noindent\textbf{Robert~I.~Jewett}

\noindent\textit{Department of Mathematics, Western Washington
University, \\ Bellingham, Washington 98225}


\begin{thebibliography}{99}

\bibitem{CJ}
B.~\'{C}urgus and R.~I.~Jewett, On the variation of $3\!\times\!3$
stochastic matrices, (August 2007), available at \\
\url{http://myweb.facstaff.wwu.edu/curgus/papers.html}.

\bibitem{Do} 
J.~L.~Doob, {\em Stochastic Processes}, John Wiley \& Sons, 1990;
reprint of the 1953 original.

\bibitem{FIS} 
S.~Friedberg, A.~Insel, L.~Spence, {\em Linear Algebra}, 4th ed.,
Prentice Hall, 2002. 


\bibitem{Ro}
N.~J.~Rose, On regular Markov chains, this {\sc Monthly}, {\bf 92}
(1985), 146.

\bibitem{S} 
E.~Seneta, {\em Non-Negative Matrices and Markov Chains}, Springer,
2006. 


 \bibitem{TRH}
 P.~Tarazaga, M.~Raydan, and A.~Hurman,  Perron-Frobenius theorem for
 matrices with some negative entries.  Linear Algebra Appl. {\bf 328}
 (2001), 57--68.

\end{thebibliography}
\end{document}